\documentclass[11pt,a4paper,reqno]{amsart}
\usepackage{amsmath,amssymb,amsfonts,epsfig,mathrsfs,cite, hyperref}
\usepackage[T1]{fontenc}
\usepackage{color}
\usepackage{array}
\usepackage{amsthm}
\usepackage{amstext}
\usepackage{graphicx}
\usepackage{setspace}

\usepackage[title]{appendix}

\usepackage{booktabs}

\usepackage{tcolorbox}

\usepackage{float}

\makeatletter
\@namedef{subjclassname@2020}{%
  \textup{2020} Mathematics Subject Classification}
\makeatother

\usepackage[margin=2.5cm]{geometry}
\usepackage{color}
\usepackage{enumitem}
\usepackage{amscd,psfrag}
\usepackage{yhmath}
\usepackage[mathscr]{eucal}

\usepackage{comment}

\allowdisplaybreaks[4]

\usepackage{slashed}

\makeatletter
\pdfpageheight\paperheight
\pdfpagewidth\paperwidth

\usepackage{epstopdf}

\usepackage{indentfirst}	

\usepackage[normalem]{ulem}
\theoremstyle{plain}

\newtheorem{definition}{Definition}
\newtheorem{theorem}[definition]{Theorem}
\newtheorem*{theorem*}{Theorem}

\newtheorem{remark}[definition]{Remark}

\newtheorem*{remark*}{Remark}
\newtheorem*{sideremark*}{Side Remark}

\newtheorem*{claim*}{Claim}
\newtheorem*{lemma*}{Lemma}
\newtheorem*{q*}{Question}
\newtheorem{lemma}[definition]{Lemma}

\newtheorem*{corollary*}{Corollary}

\newtheorem{proposition}[definition]{Proposition}

\newcommand{\R}{\mathbb{R}}

\newcommand{\na}{\nabla}

\newcommand{\id}{{\rm Id}}

\newcommand{\p}{\partial}

\newcommand{\e}{\epsilon}

\newcommand{\dd}{{\rm d}}

\newcommand{\G}{\Gamma}

\newcommand{\M}{{\mathcal{M}}}
\newcommand{\bra}{\left\langle}
\newcommand{\ket}{\right\rangle}

\newcommand{\gr}{{\bf Gr}}

\newcommand{\mres}{\mathbin{\vrule height 1.6ex depth 0pt width
0.13ex\vrule height 0.13ex depth 0pt width 1.3ex}}

\def\XXint#1#2#3{{\setbox0=\hbox{$#1{#2#3}{\int}$ }
\vcenter{\hbox{$#2#3$ }}\kern-.6\wd0}}

\newcommand{\E}{\mathbb{E}}

\newcommand{\leb}{{\mathcal{L}}}
\newcommand{\hau}{{\mathcal{H}}}

\newcommand{\hedge}{{\frac{x}{|x|}}}
\newcommand{\fp}{{\lfloor p \rfloor}}

\newcommand{\sph}{{\bf S}}

\newcommand{\ep}{{\mathcal{E}_p}}

\newcommand{\B}{{\mathbf{B}}}

\newcommand{\bn}{\B^n}
\newcommand{\snone}{\sph^{n-1}}

\newcommand{\jj}{{\bf j}}

\newcommand{\fpp}{{\frac{p}{\fp}}}

\newcommand{\D}{{\mathscr{D}}}
\newcommand{\jac}{{\mathscr{J}}}

\newcommand{\anp}{{\mathscr{A}_{n,p}}}

\title{The hedgehog uniquely minimises every $p$-energy} 

\author{Siran Li}

\address{Siran Li: School of Mathematical Sciences $\&$ CMA-Shanghai, Shanghai Jiao Tong University, No.~6 Science Buildings,
800 Dongchuan Road, Minhang District, Shanghai, China (200240)}

\email{\texttt{siran.li@sjtu.edu.cn}} 

\keywords{Harmonic map; minimising harmonic map; hedgehog; p-energy}

\subjclass[2020]{35B65, 58E20}
\date{\today}

\begin{document}

\begin{abstract}

We prove that for every dimension $n \geq 2$ and parameter $p \in ]1,n[$, the ``hedgehog'' $u_0(x) = x/|x|$ is the unique minimiser of the $p$-energy $\mathcal{E}_p [u]:=\int_{\mathbf{B}^n}|\nabla u|^p\,{\rm d}x$ in the class of $W^{1,p}$-mappings from the $n$-dimensional unit ball $\mathbf{B}^n$ to the $(n-1)$-dimensional unit sphere $\mathbf{S}^{n-1}=\partial\mathbf{B}^n$ whose trace equal to the identity on $\mathbf{S}^{n-1}$. This answers in the affirmative an open question on p.29 in [Robert M. Hardt, Singularities of harmonic maps, \textit{Bull. Amer. Math. Soc. (N.S.)} \textbf{34} (1997), 15--34].

\end{abstract}
\maketitle

\section{Introduction}

Consider the $p$-energy of a map $u$ from the $n$-dimensional unit ball $\B^n \subset \R^n$ to the unit sphere $\sph^{n-1}=\p\B^n$, with $n \geq 2$ and $1<p<n$:
\begin{equation}\label{p-energy, def}
\ep[u]:= \int_{\B^n} |\na u|^p\,\dd x,
\end{equation}
where $|\na u| = \sqrt{\sum_{1\leq i,j \leq n} \left|\frac{\p u^i}{\p x^j}\right|^2}$. When $p=2$, this is the Dirichlet energy up to a factor of $1/2$.

A long-standing question in the theory of harmonic maps is the existence and uniqueness of minimisers of $\ep$ over the admissible class
\begin{align}\label{admissible class}
    \mathscr{A}_{n,p}:=\Big\{u \in W^{1,p}(\bn;\snone):\, u\big|_{\p\bn} = \id_{\snone}\Big\}.
\end{align}
Any $u \in \anp$ has a nonempty singular set: the boundary map cannot be extended continuously to a sphere-valued map defined on $\B^n$ due to topological obstructions. The ``hedgehog'' $u_0:\B^n \to \sph^{n-1}$ is a natural candidate for an $\ep$-minimiser:
\begin{equation}\label{eq:intro-hedgehog}
  u_0(x):=\frac{x}{|x|},\qquad x \in \B^n.
\end{equation}
It is degree-$0$-homogeneous and belongs to \(\mathscr A_{n,p}\) in the subcritical range \(p<n\). It has $p$-energy
\begin{equation}\label{eq:intro-model-energy}
  \ep[u_0]   =\frac{\alpha_{n-1}}{n-p}(n-1)^{p/2},
\end{equation}
where $\alpha_j$ is the area of the $j$-dimensional unit sphere $\sph^j$.

Regarding the minimisation problem for the $p$-energy over $\anp$, Brezis--Coron--Lieb~\cite{BrezisCoronLieb1986} proved that $\mathcal{R}\left(\hedge\right)$ is a Dirichlet energy (\textit{i.e.}, $\mathcal{E}_2$) minimiser from $\B^3$ to $\sph^2$ if and only if $\mathcal{R}$ is an orthogonal rotation of $\sph^2$. See also the previous numerical results in \cite{chkll}. Almgren--Browder--Lieb~\cite{AlmgrenBrowderLieb1988} also studied the minimisation of Dirichlet energy with prescribed singularities using approaches via coarea and integral currents. J\"ager--Kaul~\cite{JagerKaul1983} proved, among other results, the $\mathcal{E}_2$-minimality of the hedgehog for $n \geq 7$, and Lin~\cite{Lin1987} subsequently proved this for arbitrary dimensions via a null-Lagrangian argument.

The case of general $p$ has been more resistant to the calibration and coarea methods applicable to $p=2$. Coron--Gulliver~\cite{CoronGulliver1989} proved the $\ep$-minimality of the hedgehog for integer exponents $p \in \{1,\ldots,n-1\}$ by exploiting projections onto low-dimensional spheres. A shorter proof was given by Avellaneda--Lin~\cite{AvellanedaLin1988} using quasi-affine integrands. Musina~\cite{Musina1994} showed that the hedgehog is $\ep$-minimising for $p \in [n-1,n[$ among degree-1 maps in $\anp$, and Hardt--Lin--Wang~\cite{HardtLinWang1998} removed the degree-1 condition by a continuation argument as $p \nearrow n$, together with the classification of normal Jacobi fields along the identity; see also~\cite{HardtLinWang1997} by the same authors on the analysis of the singularities of $\ep$-minimisers. Hong and Wang~\cite{Hong2000,Wang1998} independently extended the range of $p$ for the $\ep$-minimality to $p \in \left[2, n-2\sqrt{n-1}\right]$, using the PDE for $p$-harmonic maps. Moreover, the work~\cite{Wang1998} by Wang implies that for any $p \in ]1,n[$, if the hedgehog is an $\ep$-minimiser over $\anp$, then it is the unique one.

In the 1997 scientific report on singularities of harmonic maps, Hardt summarised the state of the problem and explicitly asked whether the hedgehog is minimising for the remaining exponents $p$~\cite[p.~29]{Hardt1997}. Hong later proposed in~\cite{Hong2001} to reduce the remaining open cases to the minimality of the hedgehog for a weighted $2$-energy. Nonetheless, Bourgoin~\cite{Bourgoin2006} exhibited ranges of $p$ in which the required weighted minimality fails, and in subsequent papers~\cite{BourgoinInduction2007, BourgoinWeighted2007} obtained several sharp weighted results for $p \in ]1,n[ \cap \mathbb{Z}$ and established an induction
principle relating weighted problems in consecutive dimensions. They do not, however,
supply the missing unweighted result for arbitrary $p \in ]1,n[ \setminus \mathbb{Z}$.

The goal of this paper is to show that the hedgehog is uniquely $\ep$-minimising for all indices $p \in ]1,n[$ in all dimensions $n \geq 2$:
 
\begin{theorem}\label{thm: main}
Given any $2 \leq n \in \mathbb{N}$ and $1<p<n$, the hedgehog $u_0(x)=\hedge$ is the unique minimiser of the $p$-energy $\ep$ in the admissible class $\anp$.  
\end{theorem}

Our proof, motivated by the technique of projecting minimisers onto low-dimensional spheres as with the case of $p \in \mathbb{Z}$ in Coron--Gulliver~\cite{CoronGulliver1989}, is of a coarea nature; \textit{cf}. Almgren--Browder--Lieb~\cite{AlmgrenBrowderLieb1988}. For general $p \in ]1,n[$, we project an $\ep$-minimiser onto low-dimensional linear subspaces $V$ that are uniformly distributed in the corresponding Grassmannian manifold, and we slice the $p$-energy of a minimiser $u \in \anp$ by first integrating along $V$ and then along $V^\perp$. Using ideas from integral geometry and/or geometric statistics (see, in particular, Lemma~\ref{lem: beta distribution}), together with variational arguments for the $p$-energy, we are able to bound $\ep$ of an arbitrary minimiser $u \in \anp$ from below by the $L^{\fpp}$-norm of the $\fp$-Jacobian of a certain projected image of $\na u$, times a constant $c_\star$ that can be evaluated (rather surprisingly) explicitly via the aforementioned geometric statistical observations. The explicit value of $c_\star$ ensures that the minimum of $\ep$ is attained precisely at the hedgehog.

\medskip
\noindent
{\bf Organisation.} The remaining parts of the paper are organised as follows:

In \S\ref{sec: lower bd}, we prove the sharp topological Jacobian lower bound. \S\S 3 $\&$ 4 present two elementary lemmata, one on linear algebra and the other on geometric statistics, which will turn out to be crucial to our proof. \S 5 constructs the normalised projections of $u \in \anp$ and discusses the key slicing arguments.
With the previous preparations, we conclude the proof of Theorem~\ref{thm: main} in \S\ref{sec: conclusion}. A detailed proof of the technical Lemma~\ref{lem: approx} is given in the appendix.

\medskip
\noindent
{\bf Notation.} Throughout this paper, we denote by $\fp$ the largest integer smaller than or equal to $p$, and for any $k \in \mathbb{N}$ we write $\alpha_k=\hau^k(\sph^k)$, \textit{i.e.}, the area of the unit $k$-sphere. Denote by $\leb^k$ and $\hau^k$ the $k$-dimensional Lebesgue and Hausdorff measures, respectively. The symbol $\mres$ designates the restriction of a measure or a function.

For a vector space $V$, we write $\B^V_r(x)$ for the open ball in $V$ centred at $x$ with radius $r$. One abbreviates $\B^V_r \equiv \B^V_r(0)$ and $\B^V = \B^V_1$; for $V=\R^k$ one puts $\B^n_r(x) = \B^{\R^n}_r(x)$; also, write $\sph(V):=\p\B^V$ for the unit sphere in $V$. For $v \in V$, write $\bra v \ket$ for the linear span of $v$ over $\R$.  The symbol $\Lambda^k$ denotes the $k$-fold exterior product.  In addition, for $u=\left(u^1, \ldots, u^n\right)^\top \in \anp$, we denote by 
$\na u$ the  $n \times n$ matrix $\left\{\p_j u^i\right\}_{1\leq i,j \leq n}$. Then  $|\na u| = \sqrt{\sum_{1\leq i,j \leq n} \left|\frac{\p u^i}{\p x^j}\right|^2}$ is the Hilbert--Schmidt norm of the $n \times n$ matrix $\na u = \{\p_j u^i\}$ for $u=\left(u^1, \ldots, u^n\right)^\top$. 

Unless otherwise specified, all the projections, norms of functions or tensors, measures, and inner products, etc., are taken with respect to the Euclidean topology. Meanwhile, for any compact Lie group $G$, the measure in consideration is the normalised Haar measures $\mu_G$ with total mass one; similarly, the Grassmannian $\gr_k(\R^n)$ is equipped with the invariant probability measure $\mu_{\gr_k(\R^n)}$ induced by the normalised Haar measure on $O(n)$. We write $|\cdot|_{\rm HS}$ for the Hilbert--Schmidt norm of a matrix, possibly with different numbers of rows and columns.

\section{$k$-Jacobian lower bound for the gradient}\label{sec: lower bd}


In this section, we establish the following integral lower bound for a linear or superlinear power of the $k$-Jacobian of $u \in \anp$. Sharpness of this bound is verified by the hedgehog.

\begin{proposition}\label{prop:sharp-k-jacobian}
Let $1 \leq k \in \mathbb{N}$, $p\in [k, k+1[$, and $r>0$. Assume that $u\in W^{1,p}\left(\B_r^{k+1}; \sph^k\right)$ has trace
$u(x)=\hedge$ for \emph{a.e.} $x\in\partial \B_r^{k+1}$.  Then
\begin{equation}\label{eq:sharp-k-jacobian}
 \int_{\B_r^{k+1}}\left|\Lambda^k (\na u(x))\right|^{\frac{p}{k}}\,\dd x \geq \frac{\alpha_k}{k+1-p}r^{k+1-p}.
\end{equation}
\end{proposition}

For our purpose, in fact, we only need to prove for the case $k<p<k+1$, as the $\ep$-minimality of the hedgehog over $\anp$ is known for $p \in \{1,2,\ldots,n-1\}$; see \cite[Theorem~2.4]{CoronGulliver1989}. When $p=k$, we are in the situation of estimating the $L^1$-norm of the $k$-Jacobian of $\na u$ for $u \in W^{1,k}$. Such bounds have appeared in Brezis--Coron--Lieb~\cite{BrezisCoronLieb1986} and Coron--Gulliver~\cite{CoronGulliver1989}, and have utilised the trick of null Lagrangian (see also Lin~\cite{Lin1987}). Here, we formulate and prove for $k\leq p <k+1$ in Proposition~\ref{prop:sharp-k-jacobian} for the sake of completeness.

The key ingredients for the proof of Proposition~\ref{prop:sharp-k-jacobian} include (i), approximating any admissible $u \in \anp$ by $W^{1,p}$-maps smooth away from only finitely many singular points which retaining the boundary datum; and (ii), estimating the $L^\fpp$-norm of a 1-form with such isolated singularities from below. We deal with (i) and (ii) in Lemma~\ref{lem: approx} and Proposition~\ref{prop:off-centre}, respectively.

\begin{lemma}\label{lem: approx}
Let $k$ be a positive integer, $p \in [k,k+1[$, and $r>0$. For any $u \in W^{1,p}(\B^{k+1}_r;\sph^k)$ satisfying $u\big|_{\p\B^{k+1}_r}(x) =\hedge$ in the sense of trace, we can find a sequence $\{u_n\} \subset W^{1,p}(\B^{k+1}_r;\sph^k)$ such that $u_n$ converges to $u$ strongly in $W^{1,p}$, $u_n(x)=\hedge$ near $\p\B^{k+1}_r$ for each $n$, and $u_n$ is smooth except at finitely many interior points for each $n$. 
\end{lemma}

The above approximation lemma has been essentially established in Bethuel~\cite{Bethuel1991}. We present a detailed proof in the Appendix. 
As a caveat, the number of singular points of $u_n$ is not necessarily uniformly bounded as $n \to \infty$, and the neighbourboods near the boundary in which $u_n$ coincides with the hedgehog are not necessarily uniform in $n$.

\begin{proposition}\label{prop:off-centre}
Let $k \in \mathbb{N}_{\geq 1}$ and $k<p<k+1$. Consider a vector field $\jj \in L^{p/k}_{\rm loc}\left(\mathbb R^{k+1};\mathbb R^{k+1}\right)$ that is smooth away from finitely many isolated singularities $\{s_1,\cdots,s_J\}$ with
\[
 \operatorname{div} \jj
 =\alpha_k\sum_{i=1}^J D_i\delta_{s_i},
 \qquad
 D_i\in\mathbb Z\setminus\{0\},
 \qquad
 \sum_{i=1}^J D_i=1.
\]
Then there exist $t\in\mathbb R^{k+1}$ and $\rho_0>0$ such that
\begin{equation}\label{eq:off-centre-estimate}
 \int_{\B_\rho(t)}|\jj|^{\frac{p}{k}}\,\dd x
 \geq \frac{\alpha_k}{\gamma}\rho^\gamma\qquad\text{for every
$\rho\geq\rho_0$}.
\end{equation}
If, in addition, $\jj(x)=\frac{x}{|x|^{k+1}}$ on $\{|x|\geq1\}$, then
\begin{equation}\label{eq:interior-j-bound}
 \int_{\B_1}|\jj|^{\frac{p}{k}}\,\dd x\geq\frac{\alpha_k}{\gamma}.
\end{equation}
\end{proposition}

All the balls in Proposition~\ref{prop:off-centre} are $(k+1)$-dimensional; $k=\fp$. We drop the superscript in $\B^{k+1}_r(x_0)$.  Also, recall that $\alpha_k$ is the $k$-dimensional area of the unit $k$-sphere $\sph^k$. One subtlety of Proposition~\ref{prop:off-centre} is that, in order to establish the bound~\eqref{eq:interior-j-bound} on the unit ball $\B_1$, we first need a bound on ``off-centre'' balls $\B_\rho(t)$ for $t\geq 0$ in the form of~\eqref{eq:off-centre-estimate}.

The proof of Proposition~\ref{prop:off-centre} has a combinatorial flavour. It is based on a ``colliding-and-merging balls'' argument, which is somewhat reminiscent of the ``Swiss cheese decomposition'' constructed in Hardt--Sullivan~\cite{hs} and later used in the study of the singularities of $\ep$-minimisers for $n-1\leq p<n$ in Hardt--Lin--Wang~\cite{HardtLinWang1998}. We start from a simple observation:

\begin{lemma}\label{lem:connected-ball-cover}
Let $\left\{\overline \B_{r_i}(a_i)\right\}_{i\in I}$ be a finite family of
closed balls whose union is connected. Then there exist $R>0$ and $a \in \R^{k+1}$ such that
\begin{equation*}
\bigcup_{i \in I}\overline \B_{r_i}(a_i) \subset \overline \B_R(a)\qquad\text{and}\qquad R\leq \sum_{i\in I}r_i.
\end{equation*}
\end{lemma}

\begin{proof}[Proof of Lemma~\ref{lem:connected-ball-cover}]
For two intersecting balls, put $d=|a_1-a_2|\leq r_1+r_2$ and choose
$a \in [a_1,a_2]$ so that $|a-a_1|=\frac{r_2}{r_1+r_2}d$ and $ |a-a_2|=\frac{r_1}{r_1+r_2}d$. Then both balls are contained in $\overline \B_{r_1+r_2}(a)$. The
general assertion follows by induction along a spanning tree of the intersection graph.  \end{proof}

\begin{proof}[Proof of Proposition~\ref{prop:off-centre}]

We divide our arguments into five steps.

\smallskip
\noindent
{\bf Step~1.} For a ball $B$ with $\p B \cap \{s_1, \cdots, s_J\} = \emptyset$, we write
\[
 Q(B):=\sum_{s_i\in B}D_i,
\]
the total charge enclosed by $B$. We say that $B$ is \emph{active} if $Q(B)\neq0$ and \emph{dormant} if $Q(B)=0$.

\smallskip
\noindent
{\bf Step~2.} If the annulus $\B_b(s)\setminus\overline \B_a(s)$ contains no charged
point, then the flux of $\jj$ through \textit{a.e.} sphere $\partial \B_r(s)$, $a<r<b$, equals $\alpha_k Q(\B_a(s))$.  Hence, by Hölder's inequality, we have
\begin{align*}
 \int_{\partial \B_r(s)}|\jj|^{\frac{p}{k}}\,\dd\hau^k \geq
 \frac{\left|\int_{\partial \B_r(s)}\jj\cdot\nu\,
 \dd\hau^k\right|^{\frac{p}{k}}}
 {(\alpha_kr^k)^{{\frac{p}{k}}-1}}=\alpha_k|Q(\B_a(s))|^{\frac{p}{k}} r^{k-p}.
\end{align*}
An integration in $r$ yields that
\begin{equation}\label{eq:annular-flux}
 \int_{\B_b(s)\setminus \B_a(s)}|\jj|^{\frac{p}{k}}\,\dd x
 \geq \frac{\alpha_k}{\gamma}|Q(\B_a(s))|^{\frac{p}{k}} (b^\gamma-a^\gamma).
\end{equation}
Recall $\gamma = k+1-p \in ]0,1[$.

\smallskip
\noindent
{\bf Step~3.} Now, let us choose an initial family of balls  $\mathcal{B}_{\rm init} = \{\B_{r_i}(s_i)\}_{i \in I}$ by requiring that $r_i$'s are so small that these balls are disjoint. Thus, each ball contains 
no charged point other than its centre.  We deduce from~\eqref{eq:annular-flux} by taking $b=r_i$, $a \searrow 0$, and $s=s_i$ that
\[
 \int_{\B_{r_i}(s_i)}|\jj|^{\frac{p}{k}}\,\dd x
 \geq \frac{\alpha_k}{\gamma}|D_i|^q r_i^\gamma
 \geq \frac{\alpha_k}{\gamma}r_i^\gamma.
\]

Given a finite collection $\mathcal{B}$ of balls. For each $B \in \mathcal{B}$, we say that it carries a \emph{credit set} $K_B$ if $\{K_B\}_{B \in\mathcal{B}}$ are mutually disjoint, $K_B \subset B$, and  
\begin{equation}\label{eq:credit-invariant}
 \int_{K_B}|\jj|^{\frac{p}{k}}\,\dd x \geq\frac{\alpha_k}{\gamma}\,[{\rm rad}(B)]^\gamma\qquad\text{for each } B \in \mathcal{B}.
\end{equation}
The initial balls $\mathcal{B}_{\rm init}$ themselves constitute a collection of credit sets.

\smallskip
\noindent
{\bf Step~4.} Now, fix each dormant ball in $\mathcal{B}_{\rm init}$ and enlarge one active ball $B_{\rm act}$ concentrically until it meets another ball for the first time. Before this collision, the enlarged ball (call it $\widehat{B}_{\rm act}$) will have no new charged point. Since its charge is a nonzero integer, we have $|Q(B_{\rm act})|^{p/k}\geq 1$, so \eqref{eq:credit-invariant} is preserved by taking the credit set to be the enlarged ball.

At the time that the enlarged ball collide with some other ball(s), by Lemma~\ref{lem:connected-ball-cover} we may replace the connected cluster of touching balls by a ball ($\widetilde{B}_{\rm repl}$, say) containing them whose radius $R$ is at most the sum of their radii. If $\widetilde{B}_{\rm repl}$ meets other balls, absorb those
as well and repeat this process. At each such ``merger'', we define the new credit set $K_{\rm new}$ to be the union of the old credit sets. That is, whenever a collection of balls collides, their corresponding contributions to the credit set gets fixed once and for all. (The replaced cluster may continue to grow and collide, but their parts of the credit set will change no more.)

In view of the above construction, from the monotonicity of credit sets, the concavity of $s\mapsto s^\gamma$ ($0<\gamma<1$), and that $R\leq \sum_{i\in I}r_i$ by Lemma~\ref{lem:connected-ball-cover}, we deduce that
\begin{equation}\label{credit, xx}
 \int_{K_{\rm new}}|\jj|^{\frac{p}{k}}\,\dd x
 \geq \frac{\alpha_k}{\gamma}\sum_i r_i^\gamma
 \geq \frac{\alpha_k}{\gamma}
       \left(\sum_i r_i\right)^\gamma
 \geq \frac{\alpha_k}{\gamma}R^\gamma.
\end{equation}
In particular, notice that the bound~\eqref{eq:credit-invariant} is preserved, regardless of whether the charge of the merged ball is zero.

We may repeat this above process and arrive at one single ball in finitely many steps. This is because the sum of the charges of all current balls is $1$, so there must be an active ball whenever more than one ball remains. We enlarge an active ball until the next collision and merge. Each merger strictly
decreases the number of balls, so after finitely many mergers, there is
one remaining ball $B_{\rho_0}(t)$ with charge 1.

Now, the desired estimate~\eqref{eq:off-centre-estimate} holds with $\rho = \rho_0$ in view of the credit already accumulated; see~\eqref{credit, xx}. For any $\rho \geq\rho_0$, we keep the centre $t$ fixed and enlarge the ball $B_{\rho_0}(t)$ to $B_{\rho}(t)$. Thus, \eqref{eq:off-centre-estimate} holds by the case $\rho=\rho_0$ and the formula~\eqref{eq:annular-flux}, which takes into account the contributions from the annular region without other singularities.

\smallskip
\noindent
{\bf Step~5.} Finally, let us prove the ``in addition'' part. Let $t$ and $\rho_0$ be as in the previous step. Since $\B_\rho(t)\subset \B_{\rho+|t|}(0)$, for every
sufficiently large $\rho$ we have that
\begin{align*}
 \frac{\alpha_k}{\gamma}\rho^\gamma
 &\leq \int_{\B_{\rho+|t|}(0)}|\jj|^{\frac{p}{k}}\,\dd x \\
 &=\int_{\B_1}|\jj|^{\frac{p}{k}}\,\dd x
   +\alpha_k\int_1^{\rho+|t|}r^{k-p}\,\dd r \\
 &=\int_{\B_1}|\jj|^{\frac{p}{k}}\,\dd x
   +\frac{\alpha_k}{\gamma}
      \bigl((\rho+|t|)^\gamma-1\bigr).
\end{align*}
As $0<\gamma<1$, it holds that
$(\rho+|t|)^\gamma-\rho^\gamma\to0$ as $\rho\to\infty$. Thus~\eqref{eq:interior-j-bound} follows.

The proof of Proposition~\ref{prop:off-centre} is now complete.   \end{proof}

Equipped with Lemma~\ref{lem: approx} and Proposition~\ref{prop:off-centre} established just now, we are at the stage of showing the main result of this section.

\begin{proof}[Proof of Proposition~\ref{prop:sharp-k-jacobian}]
By scaling, we only need to prove for $r=1$.

\smallskip
\noindent
\underline{\bf Case $1<p<n$, $p \notin \mathbb{Z}$.} Denote $k=\fp$ as before. Apply Lemma~\ref{lem: approx} to choose a sequence $\{u_\ell\}$ such that $u_\ell \to u$ strongly in $W^{1,p}(\B_1^{k+1};\sph^k)$, each $u_\ell$ equals $x/|x|$ in a collar of the boundary, and each $u_\ell$ is smooth away from finitely many interior singularities.  Extend $u_\ell$ to
$\widetilde u_\ell:\mathbb R^{k+1}\to \sph^k$ by the hedgehog outside $\B^{k+1}_1$.

Let $\omega_{\sph^k}$ denote the standard  Riemannian volume form on $\sph^k$, so that
$\int_{\sph^k}\omega_{S^k}=\alpha_k$. Set
\[
 \jj_\ell:=\varepsilon_k
 \bigl(\star\widetilde u_\ell^{\#}\omega_{\sph^k}\bigr)^\sharp,
\]
where the fixed sign $\varepsilon_k\in\{-1,1\}$ is chosen so that the
hedgehog produces the outward field $x/|x|^{k+1}$.
The Hodge star $\star$ is taken in $\mathbb R^{k+1}$ and the superscript ${}^\#$ denotes pullback, so
$\star\widetilde u_\ell^{\#}\omega_{S^k}$ is a 1-form. The sharp sign $\sharp$ is the canonical isomorphism turning a 1-form to a vector field.  Since $\na \widetilde u_\ell$ takes values in the $k$-plane $T_{\widetilde u_\ell}\sph^k$, by Cauchy--Binet formula we have that
\[
 \left|\jj_\ell\right|=\left|\Lambda^k \na\widetilde u_\ell\right|.
\]
The form $\widetilde u_\ell^{\#}\omega_{\sph^k}$ is closed away from
the finitely many singularities, while the degree formula on small surrounding
spheres therefore yields that
\[
 \operatorname{div}\jj_\ell
 =\alpha_k\sum_i D_{\ell,i}\delta_{s_{\ell,i}},
 \qquad D_{\ell,i}\in\mathbb Z\setminus\{0\}.
\]
There is no boundary contribution because $u_\ell$ agrees with the
hedgehog in a collar.  Moreover,
\[
\jj_\ell(x)=\frac{x}{|x|^{k+1}}\qquad\text{ for } |x|\geq 1,
\]
so its flux through every sufficiently large sphere is $\alpha_k$; hence, $\sum_iD_{\ell,i}=1$. Thus, by Proposition~\ref{prop:sharp-k-jacobian}, we have that
\[
 \int_{\B^{k+1}_1}\left|\Lambda^k \na u_\ell\right|^{\frac{p}{k}}\,\dd x
 =\int_{\B^{k+1}_1}\left|\jj_\ell\right|^{\frac{p}{k}}\,\dd x
 \geq\frac{\alpha_k}{k+1-p}.
\]

Finally, the multilinear estimate $\left|\Lambda^kA-\Lambda^kB\right|
 \lesssim_k |A-B|\bigl(|A|^{k-1}+|B|^{k-1}\bigr)$ and Hölder's inequality imply that
\[
 \Lambda^k \na u_\ell\longrightarrow\Lambda^k \na u
 \quad\text{strongly in }L^{p/k}(\B^{k+1}_1).
\]
We may now conclude the proof of Proposition~\ref{prop:sharp-k-jacobian} in the case $p \notin \mathbb{Z}$ by sending $\ell \to \infty$.

\smallskip
\noindent
\underline{\bf Case $p\in\{1,2,\ldots,n-1\}$.}  To this end, we first collect a well-known lemma, which states that the cofactor matrix of a gradient is a null Lagrangian. The proof is omitted. 

\begin{lemma}
\label{lem:null-lagrangian-cofactor}
Let \(\Omega\subset\mathbb R^{k+1}\) be a bounded Lipschitz domain and let 
$v,w\in W^{1,k}(\Omega;\mathbb R^{k+1})$ be of the same trace on \(\partial\Omega\). Then
\begin{equation*} 
\int_\Omega \operatorname{tr}(\operatorname{cof} \na v)\,\dd x
=
\int_\Omega \operatorname{tr}(\operatorname{cof}\na w)\,\dd x .
\end{equation*}
Here \(\operatorname{cof}A=D(\det)(A)\), so that $
A^*\operatorname{cof}A
=
(\operatorname{cof}A)^*A
=
(\det A)\,\id_{k+1}.$
\end{lemma}

As \(u\) is sphere-valued, $u \cdot \na u = 0$ \textit{a.e.}. We \emph{claim} that, for some measurable vector field
\(\jj:\B_r^{k+1}\to\mathbb R^{k+1}\), one has
\begin{equation}
\label{eq:cofactor-rank-one}
\operatorname{cof}\na u=u\otimes \jj
\qquad\text{a.e.}.
\end{equation}
Indeed, if \(\operatorname{rank} \na u\leq k-1\), then
\(\operatorname{cof} \na u=0\). If
\(\operatorname{rank} \na u=k\), then $
\ker (\na u)^*=\operatorname{span}\{u\} \equiv \bra u\ket$. On the other hand,
 $(\na u)^*\operatorname{cof} \na u
=
(\det \na u)\id_{k+1}=0,$ so every column of \(\operatorname{cof} \na u\) belongs to $\bra u\ket$. This proves~\eqref{eq:cofactor-rank-one}; in particular, $\operatorname{cof} \na u$ is of rank one. In fact, one may take $
j=(\operatorname{cof}\na u)^* u$.

The entries of \(\operatorname{cof} \na u\) are,
up to signs, the \(k\times k\) minors of \(\na u\). Hence $|\operatorname{cof} \na u|_{\mathrm{HS}} = |\Lambda^k \na u|.$
Since \(|u|=1\), the \emph{claim}~\eqref{eq:cofactor-rank-one} also yields that  $|\operatorname{cof} \na u|_{\mathrm{HS}}=|\jj|$ and $\operatorname{tr}(\operatorname{cof}\na u)=u\cdot \jj$. We thus have the following pointwise calibration condition:
\begin{equation}
\label{eq:pointwise-calibration}
|\Lambda^k \na u|
=
|\jj|
\geq u\cdot \jj
=
\operatorname{tr}(\operatorname{cof}\na u).
\end{equation}

Now, by virtue of Lemma~\ref{lem:null-lagrangian-cofactor}, we have \begin{equation}
\label{eq:null-application}
\int_{\B_r^{k+1}}
\operatorname{tr}(\operatorname{cof}\na u)\,\dd x
=
\int_{\B_r^{k+1}}
\operatorname{tr}(\operatorname{cof}\na u_0)\,\dd x,
\end{equation}
where $u_0$ is the hedgehog. In polar coordinates \(x=\rho\omega\) with $\rho>0$ and $\omega\in\sph^k$, we have $\na u_0(x)
=
\frac1{\rho}\bigl(\id-\omega\otimes\omega\bigr)$ and hence $\operatorname{tr}(\operatorname{cof}\na u_0(x))
=
\rho^{-k}$. Thus,
\begin{align}\label{aa}
\int_{\B^{k+1}_r}\operatorname{tr}(\operatorname{cof}\na u_0)\,\dd x = 
\int_0^r\int_{\sph^k}
\rho^{-k}\rho^k\,\dd\hau^k(\omega)\,\dd\rho
=\alpha_k r.
\end{align}
Inequality~\eqref{eq:sharp-k-jacobian} for $p=k \in \mathbb{N}$ follows directly from~\eqref{eq:pointwise-calibration}, 
\eqref{eq:null-application}, and \eqref{aa}.

We conclude the proof of Proposition~\ref{prop:sharp-k-jacobian} by combining the above two cases.   \end{proof}

\section{A lemma on linear algebra}\label{sec: linear alg lemma}

Let us present an elementary linear algebraic lemma in a fairly general form. In later parts, we shall apply it with $m=n-1$, $k=\fp$, and $V' = \vartheta^\perp$ for some given unit vector $\vartheta \in \sph^{n-1}$. We write $|\cdot|_{\rm HS}$ for the Hilbert--Schmidt norm of a matrix.

\begin{lemma}\label{lem: linear algebra}
Fix an $m$-dimensional vector space $V'$ of $\R^n$. Let $H$ be randomly sampled from  $\gr_{k}\left(V'\right)$ with respect to the Haar measure. The expectation of the $k^{\text{th}}$ Jacobian determinant of a linear transform $T:\R^n \to V'$ restricted to $H$ can be estimated as follows: \begin{align*}
\E_H\left[  \Big(\jac_k\left(T^*\mres H\right)\Big)^2 \right] \leq \left( \frac{|T|_{\rm HS}^2}{m} \right)^k.
\end{align*}
For $k=1$, this is always an equality; while for $k \geq 2$, equality holds if and only if $TT^*$ is a constant multiple of the identity matrix $\id_{m}$.
\end{lemma}

Recall the notation for the $k$-Grassmannian of $V'$: $$\gr_k(V'):=\Big\{ W' \subset V':\, W' \text{ is a $k$-dimensional subspace of $V'$} \Big\}.$$ It is a compact manifold on which the orthogonal group $O(V')$ acts transitively, topologised by the distance $d(W_1', W_2'):=\left\|\pi_{W_1'} - \pi_{W_2'}\right\|$, the operator norm of the difference between orthogonal projections of $V'$ onto $W_1'$ and $W_2'$. At times, we refer to $W' \in \gr_k(V')$ as a $k$-plane in $V'$.

Also, note that for any index $\gamma' \in ]0,1[$, we have by Jensen's inequality
\begin{align*}
\E_H\left[ \Big(\jac_k\left(T^*\mres H\right)\Big)^{2\gamma'} \right]\leq \left\{\E_H\left[ \Big(\jac_k\left(T^*\mres H\right)\Big)^{2} \right]\right\}^{\gamma'} \leq \left( \frac{|T|_{\rm HS}^2}{m} \right)^{\gamma' k}.
\end{align*}
This shall be later applied with $\gamma' = \frac{p}{2\fp}$.

\begin{proof}[Proof of Lemma~\ref{lem: linear algebra}]
Let $\lambda_1 \geq \cdots  \geq \lambda_{m}\geq 0$ be the eigenvalues of $TT^* \in {\rm End}(V')$. For each index set $I \subset \{1,\ldots,m\}$ with $|I|=k$, denote by $M_I$ the $k\times k$-minor obtained by selecting rows of $\eta_H$ indexed by $I$, where $\eta_H: \R^k \to \R^n$ is the inclusion of a $k$-plane with image ${im}\,\eta_H = H \subset V'$. Then, in view of the identity $$\sqrt{\det\left[\left(\pi_HT T^*\pi_H^*\right)\mres H\right]} = \jac_k\left(T^*\mres H\right)$$  and the Cauchy--Binet formula, we have
\begin{align}\label{jac, 2}
\Big(\jac_k\left(T^*\mres H\right)\Big)^2 = \sum_{I \subset \{1,\ldots,m\},\, |I|=k} \left( \prod_{i\in I} \lambda_i \right)\left(\det M_I \right)^2.
\end{align}

Since $H$ is uniformly distributed with respect to the Haar measure on $\gr_k\left(V'\right)$, the expectation $\E_H\left[\left(\det M_I \right)^2\right]$ is independent of the choice of $I$. By Cauchy--Binet again and the definition of $\eta_H$, we have
\begin{align*}
1 = \det(\eta^*_H\eta_H) = \sum_{I \subset \{1,\ldots,m\},\, |I|=k} \det(M_I)^2,
\end{align*}
so $$\E_H\left[\left(\det M_I \right)^2\right]= {{m}\choose{k}}^{-1}.$$ Note that ${{m}\choose{k}}$ is the number of choices for the index set $I$. Then, in light of \eqref{jac, 2}, we obtain that
\begin{align}\label{jac, 3}
\E_H\left[ \Big(\jac_k\left(T^*\mres H\right)\Big)^2 \right] = \frac{ \sum_{I \subset \{1,\ldots,,\},\, |I|=k}  \prod_{i\in I} \lambda_i }{{{m}\choose{k}}}.
\end{align}

The numerator on the right-hand side of \eqref{jac, 3} satisfies
\begin{align*}
 \sum_{I \subset \{1,\ldots,m\},\, |I|=k}  \prod_{i\in I} \lambda_i  = \sigma_k(\lambda_1, \ldots,\lambda_{m}),
\end{align*}
where $\sigma_k$ designates the $k$-th symmetric polynomial. Maclaurin's inequality ascertains that
$$\mathfrak{M}_1 \geq \mathfrak{M}_2 \geq \cdots \geq \mathfrak{M}_m,$$
where for any nonnegative numbers $\lambda_1, \cdots, \lambda_m$, $$\mathfrak{M}_k:=\left[ \frac{\sigma_k(\lambda_1, \ldots, \lambda_{m})}{{m\choose k}}\right]^{1/k}.$$ Moreover, for $k \geq 2$, equalities are attained if and only if all the $\lambda_j$'s are equal.

Therefore, from \eqref{jac, 3} one infers that
\begin{align*}
\E_H\left[\Big(\jac_k\left(T^*\mres H\right)\Big)^2 \right] \leq \left(\frac{\lambda_1+\cdots+\lambda_{m}}{{{m}\choose k}}\right)^k.
\end{align*}
This proves Lemma~\ref{lem: linear algebra} since $\lambda_1+\cdots+\lambda_{m} = {\rm trace}(TT^*) = |T|_{\rm HS}^2$.    \end{proof}

\section{A lemma of geometric statistics}\label{sec: stats}

In this section, we analyse the distribution of the length of the projection of a unit vector in $\R^n$ onto a randomly sampled $d$-plane. Throughout, we write $\pi_V: \R^n \to V$ as the Euclidean orthogonal projection onto a subspace $V \subset \R^n$. For subsequent developments, it is crucial that the expectation of the negative moments of this length can be explicitly calculated.

\begin{lemma}\label{lem: beta distribution}
Fix any $\vartheta \in \sph^{n-1}$; $2 \leq n \in \mathbb{Z}$. Let $V$ be randomly sampled from $\gr_d(\R^n)$ with $d<n$. Then the length of the $V$-component of $\vartheta$, namely $s := \left|\pi_V\vartheta\right|$, satisfies
\begin{equation}\label{beta distr}
s^2 \sim {\rm Beta}\left( \frac{d}{2}, \frac{n-d}{2} \right).
\end{equation}  
Moreover, it holds for any $p<d$ that
\begin{equation}\label{E s-p}
\E_V\left[s^{-p}\right] = \frac{ \G\left(\frac{n}{2}\right)\G\left(\frac{d-p}{2}\right)}{\G\left(\frac{d}{2}\right)\G\left(\frac{n-p}{2}\right)}.
\end{equation}
\end{lemma}

For $n-1\leq p <n$ we have $n=d$, so the right-hand side of \eqref{beta distr} is not well-defined. However, in this case $V=\R^n$, $s = |\vartheta|  \equiv 1$, and $\E_V[s^{-p}]=1$, so the identity~\eqref{E s-p} remains valid.

In the proof below, we shall use the fact that if a random variable $Z\sim {\rm Beta}(a,b)$, then for any $q>-a$ it holds that \begin{align*}
\E[Z^q] = \frac{B(a+q,b)}{B(a,b)}
\end{align*}
with $B(x,y):= {\G(x)\G(y)}/{\G(x+y)}$.

\begin{proof}[Proof of Lemma~\ref{lem: beta distribution}]

Let us write $V_0 = \R^d \oplus \{0^{n-d}\} \subset \R^n$.

Since $V$ is uniformly distributed on $\gr_d(\R^n)$, we have $V= RV_0$ with $R$ uniformly distributed on $O(n)$, both with respect to the Haar measures. Hence,
\begin{align*}
s^2 = \left|\pi_{RV_0}\vartheta\right|^2 = \left|\pi_{V_0} \left(R^\top \vartheta\right)\right|^2,
\end{align*}
where $R^\top \vartheta$ is randomly sampled from $\sph^{n-1}$. 

Denote a generic point on $\sph^{n-1}$ as ${(z_1, \ldots, z_n)}/{\sqrt{z_1^2+\cdots + z_n^2}}$. Then we may express
\begin{align*}
s^2 = \frac{z_1^2 + \cdots + z_d^2}{z_1^2+\cdots + z_n^2}\qquad \text{ with } z_j \sim \mathcal{N}(0,1).
\end{align*}
Here $X = z_1^2 + \cdots + z_d^2 \sim \chi_d^2$ and $Y = z_{d+1}^2 + \cdots + z_n^2 \sim \chi_{n-d}^2$, both obeying chi-square distributions. Then the law for $s^2$ is a beta distribution:
\begin{align*}
s^2 = \frac{X}{X+Y} \sim {\rm Beta}\left( \frac{d}{2}, \frac{n-d}{2}\right).
\end{align*}
This completes the proof of Lemma~\ref{lem: beta distribution}.  \end{proof}

We also observe the following:
\begin{remark}\label{rem: independence}
If $V$ is randomly sampled over $\gr_d(\R^n)$, then $H = V \cap \vartheta^\perp$  is randomly sampled over $\gr_{d-1}\left(\vartheta^\perp\right)$. Indeed, as the law of $V$ is $O(n)$-invariant,  the  law of $H$ is $\mathcal{K}$-invariant, where $$\mathcal{K}:=\left\{g\in O(n):\, g\vartheta=\vartheta \right\} = O\left(\vartheta^\perp\right) \cong O(n-1),$$ which acts on $\gr_{d-1}\left(\vartheta^\perp\right)$ transitively. In addition, the random variable $H\in \gr_{d-1}\left(\vartheta^\perp\right)$ is independent of $s = |\pi_V\vartheta|$.
\end{remark}

\section{Slicing over randomly sampled subspaces}

We now embark on the proof of the main Theorem~\ref{thm: main}. The idea is to bound the $p$-energy $\ep[u]$ of $u \in \anp$ from below:
\begin{align*}
\ep[u] \geq c_{n,p} \iint \left|\jac_{\fp}(\na_V w_V(x))\right|^{\fpp}\,\dd x\,\dd V,
\end{align*}
with the integral taken over $(x,V) \in \B^n \times \gr_{\fp+1}(\R^n)$; see~\eqref{vip, a} for the precise formulation. Here, the mapping $w_V$ is just the projection $\pi_Vu$ normalised by its length, with $V$ uniformly distributed over the Grassmannian $\gr_{\fp+1}(\R^n)$; the symbol $\na_V$ means the gradient taken only along the vectors in the linear space $V$. Observe that $w_V$, viewed as a function defined on a ball in $V$, is again a sphere-valued $W^{1,p}$-mapping whose trace coincides with that of the hedgehog (Lemma~\ref{lem: wV}). Then, from Proposition~\ref{prop:sharp-k-jacobian} we infer the bound~\eqref{vip, b},   which reads schematically: 
\begin{align*}
 \iint \left|\jac_{\fp}(\na_V w_V(x))\right|^{\fpp}\,\dd x\,\dd V \geq c'_{n,p}.
\end{align*}

Combining the two estimates above, we deduce that $\ep[u] \geq c_{n,p} \times c'_{n,p}$. It is crucial that the constants $c_{n,p}$ and $c'_{n,p}$ are explicit; in fact, their product is precisely the $p$-energy of the hedgehog. This proves the minimality of $\hedge$.

\subsection{Geometric set up}\label{subsec: geometric setup}

For the above purpose, let us first introduce the geometric setting for our slicing argument. In what follows, we take $n \geq 2$ and $1<p<n$  as in Theorem~\ref{thm: main}, and set
\begin{align*}
2 \leq d := \fp+1 \leq n,\qquad 0<\gamma := d-p \leq 1.
\end{align*}

Given an $n$-dimensional unit vector $\vartheta \in \sph^{n-1}$ and a $d$-plane $V$ in $\R^n$, we take $H$ to be the intersection of $V$ with the orthogonal complement of the span of $\vartheta$:
\begin{equation}\label{def, V and H}
V \in \gr_d(\R^n), \qquad H := V \cap \vartheta^\perp. 
\end{equation}
In particular, $V=\R^n$ when $n-1 \leq p <n$. Observe that with $\vartheta$ given, $\dim H = \dim V -1 = \fp$ for \textit{a.e.} $V$ with respect to the Haar measure on the Grassmannian $\gr_d(\R^n)$.   Here and hereafter, all the projections, norm of functions or tensors, measures, and inner products, etc., are Euclidean unless otherwise specified. We also write $\eta_H: \R^\fp \to \R^n$ for the inclusion of a $\fp$-plane with image ${im}\,\eta_H = H \subset V \subset \R^n$.

Consider a linear map $T \in {\rm Hom}\left(\R^n; \vartheta^\perp\cong\R^{n-1}\right)$. Then set
\begin{equation}\label{def: M, S}
\begin{cases}
TT^* \in {\rm End}\left(\vartheta^\perp\right),\\
S:= \pi_H T \pi_V \in {\rm Hom}\left(\R^n; H\right).
\end{cases}
\end{equation} 
It holds that $SS^* \leq \pi_H  TT^* \pi_H$ on $H$, as square matrices or quadratic forms of size $\fp \times \fp$. Using the monotonicity of determinant, we deduce that
\begin{equation*}
\left|\Lambda^\fp S\right|^2 \leq \det\left(\pi_H TT^*\pi_H \mres H\right) = \Big(\jac_\fp \left(T^*\mres H\right)\Big)^2.
\end{equation*}  

The geometric layout above shall be specialised to 
\begin{align*}
\vartheta = u(x)\qquad\text{and}\qquad T = \na u (x).
\end{align*}
When $u \in W^{1,p}\left(\B^n; \sph^{n-1}\right)$, $T$ is a linear mapping from $\R^n$ to $\vartheta^\perp$ understood as $$T\xi = \xi \cdot \na u(x) = D_\xi u(x),$$ which is valid for \textit{a.e.} $x \in \B^n$.

\subsection{Slicing of $\B^n$} 

We now slice the $n$-dimensional unit ball by randomly distributed affine $d$-planes $V \in \gr_{d}(\R^n)$, where $d = \fp+1 \in [2,n]$ as before.

For each $x \in \R^n$, one has the unique decomposition
\begin{align}\label{y, xi decomposition}
x = y+\xi \qquad\text{where } y \in V^\perp \text{ and } \xi \in V.
\end{align}
Given $y \in V^\perp$,  consider the projection of the intersection of $\B^n$ with the affine space $y+V$ onto the subspace $V$.  In view of the Pythagorean theorem and~\eqref{y, xi decomposition}, this is nothing but the ball in the $d$-dimensional subspace $V$ centred at $0$ with radius $\sqrt{1-|y|^2}$. That is, 
\begin{equation*}
\B^{V}_{\sqrt{1-|y|^2}} = \Big\{\xi \in V:\, y+\xi \in \B^n\Big\}.
\end{equation*}

Also note by Fubini's theorem that 
\begin{align}\label{fubini}
\int_{\B^n}f(x)\,\dd x = \int_{\B^{V^\perp}_1} \int_{\B^V_{\sqrt{1-|y|^2}}}f(y+\xi)\,\dd\xi\,\dd y
\end{align} 
for each integrable function $f$ on $\B^n$.

\subsection{An associated hedgehog-like map}

For $u\in \anp$, consider a hedgehog-like map associated to it:
\begin{equation}\label{wV, def}
w_V(x):= \frac{\pi_Vu(x)}{|\pi_Vu(x)|}\qquad\text{for } V \in \gr_d(\R^n).  
\end{equation}For ease of notation, we write $$\rho(y):=\sqrt{1-|y|^2}\qquad \text{ for $y \in V^\perp.$}$$ Recall also $y,\xi$ from the orthogonal decomposition~\eqref{y, xi decomposition}. 

\begin{lemma}\label{lem: wV}
Given $u \in \anp$. The map $w_V$ defined in \eqref{wV, def} lies in $W^{1,p}\left(\B^n; \sph(V)\right)$ for \emph{a.e.} $V \in \gr_d(\R^n)$. Its gradient with respect to $V$-directions is given by
\begin{align*}
\na_Vw_V = \frac{\pi_H (\na u) \pi_V}{|\pi_V u|}.
\end{align*} 
Moreover, $w_{V,y}(\xi):=w_V(y+\xi)$ lies in $W^{1,p}\left(\B^V_{\rho(y)}; \sph(V) \right)$, with
\begin{align*}
\na_\xi w_{V,y}(\xi) = \na_Vw_V(y+\xi)\qquad \text{and}\qquad w_{V,y}\Big|_{\p\B^V_{\rho(y)}}(\xi) = \frac{\xi}{|\xi|}
\end{align*}
for \emph{a.e.} $y \in V^\perp$.

\end{lemma}

This is nontrivial as $f \in W^{1,p}$ and $f \neq 0$ \textit{a.e.} alone cannot ensure $f/|f| \in W^{1,p}$. As usual, the boundary value of $ w_{V,y}$ is understood in the sense of trace. Recall that $\mu_{\gr_d(\R^n)}$ is the normalised Haar measure on $\gr_d(\R^n)$.

\begin{proof}[Proof of Lemma~\ref{lem: wV}]

To see that $w_V$ is well defined, we first check that its denominator is nonzero \textit{a.e.}.  Indeed, given $\vartheta \in \sph^{n-1}$, we have $\pi_V\vartheta=0$ if and only if $V \subset \vartheta^\perp$, so $$\mu_{\gr_d(\R^n)}\left(\left\{V \in \gr_d(\R^n):\,\pi_V\vartheta=0 \right\}\right)=0.$$ Taking $\vartheta = u(x)$, we deduce from Fubini--Tonelli that
\begin{align*}
0 = \int_{\gr_d(\R^n)} \leb^n\Big( \left\{x \in \B^n:\,\pi_V u(x)=0 \right\}\Big)\,\dd \mu_{\gr_d(\R^n)}(V).
\end{align*}
Thus, for $\mu_{\gr_d(\R^n)}$-\textit{a.e.} $V \in \gr_d(\R^n)$ and $\leb^n$-\textit{a.e.} $x \in \B^n$, we have $\pi_Vu(x)\neq 0$.

To proceed, for each given $V \in \gr_d(\R^n)$, we consider
\begin{align*}
w_{V,\e}(x):= \frac{\pi_Vu(x)}{\max\{|\pi_V u(x)|,\e\}}.
\end{align*}
Clearly $w_{V,\e} \in W^{1,p}\left(\B^n; V\right)$ since $u \in \anp$, and one also has the pointwise bound: $$|\na w_{V,\e}| \leq \frac{|\pi_V\na u|}{|\pi_V u|}.$$ Once we show that the right-hand side is in $L^p$, it holds by the Lebesgue dominated convergence theorem and the \textit{a.e.} convergence $w_{V,\e} \to w_V$ as $\e \searrow 0$ that $w_V \in W^{1,p}$. The identity $\na_\xi w_{V,y}(\xi) = \na_Vw_V(y+\xi)$ follows immediately from the definition of $w_{V,y}$, and its boundary condition follows from that of $u \in \anp$.

It remains to check that 
\begin{align*}
\frac{|\pi_V\na u|}{|\pi_V u|} \in L^p(\B^n)\qquad\text{for a.e. $V \in \gr_d(\R^n)$}.
\end{align*}
To this end, recall from the geometrical statistical Lemma~\ref{lem: beta distribution} that 
\begin{align*}
\E_V\big[ |\pi_Vu(x)|^{-p} \big] = \int_{\gr_d(\R^n)} |\pi_Vu(x)|^{-p}\,\dd \mu_{\gr_d(\R^n)}(V)  = c_{n,p} < \infty. 
\end{align*}
for $\leb^n$-\textit{a.e.} $x \in \B^n$; here $c_{n,p}=\frac{ \G\left(\frac{n}{2}\right)\G\left(\frac{d-p}{2}\right)}{\G\left(\frac{d}{2}\right)\G\left(\frac{n-p}{2}\right)}$. Thus, using the pointwise bound $|\pi_V(\na u)| \leq |\na u|$, we deduce that
\begin{align*}
\int_{\gr_d(\R^n)} \left(\frac{|\pi_V\na u|}{|\pi_V u|}\right)^p \,\dd \mu_{\gr_d(\R^n)}(V) \leq c_{n,p} |\na u|^p.
\end{align*}
The right-hand side is integrable over $\B^n$ since $u \in W^{1,p}$. 
 This proves Lemma~\ref{lem: wV}.   \end{proof}

Moreover, for each $\vartheta =u(x) \in \sph^{n-1}$, $V$ randomly sampled from $\gr_d(\R^n)$, and $H:= V \cap \vartheta^\perp$ as above, it holds that
\begin{equation}\label{decomposition of V}
V = H \oplus \bra w_V \ket.
\end{equation} 
Indeed, $u \cdot w_V = \frac{u\cdot \pi_Vu}{|\pi_Vu|} = |\pi_Vu| \neq 0$ for \textit{a.e.} $(x,V)$.  As $w_V$ is of unital length, we have the decomposition:
\begin{equation*}
\id_V = \pi_H + w_V \otimes w_V,
\end{equation*}
in the sense that $v=\pi_H(v) + (w_V\cdot v)w_V$ for every $v \in V$.

\section{Conclusion}\label{sec: conclusion}

With the preparations in the preceding sections, we are now ready to prove our main Theorem~\ref{thm: main}. We first show that the hedgehog $u_0(x)=\hedge$ is an $\ep$-minimiser over $\anp$ for any $1<p<n \in \mathbb{Z}_{\geq 2}$, and then prove its uniqueness.

\subsection{Minimality of hedgehog}\label{subsec: min}

We first recast the computations for $\na_V w_V$, the $V$-gradient of the associated hedgehog, into the earlier framework in \S\S\ref{sec: linear alg lemma} and \ref{subsec: geometric setup}. Here, as before, $d=\fp+1$, $\gamma=d-p$, $\vartheta = u(x) \in \sph^{n-1}$, $T = \na u (x) \in {\rm Hom}(\R^n; \vartheta^\perp)$, $V \in \gr_d(\R^n)$, and $H = V \cap \vartheta^\perp$.  Everything in this subsection is understood in the $\leb^n$-\textit{a.e.} sense of $x \in \B^n$ unless otherwise specified, and we often suppress the $x$ variable to unburden the notation.

In view of Lemma~\ref{lem: wV}, 
we have
\begin{align*}
\na_V w_V = \frac{S}{s},\qquad s = |\pi_VT|,
\end{align*}
where $S = \pi_H T\pi_V  \in {\rm Hom}(\R^n;H)$ as in~\eqref{def: M, S}. Since $\pi_H TT^* \pi_H - \pi_H T\pi_V T^*\pi_H \geq 0$ as square matrices, we have by the monotonicity of determinant that
\begin{align*}
 \left[\jac_\fp(\pi_H T \pi_V)\right]^2 \leq \det \left(\pi_H TT^* \pi_H\right),
\end{align*}
and hence by Jensen's inequality,
\begin{align}\label{pointwise, jac est}
\left(\jac_\fp \left( \na_V w_V \right) \right)^\fpp \leq s^{-p} \left[\det \left(\pi_H TT^* \pi_H\right)\right]^{\frac{p}{2\fp}}.
\end{align}

Integrating over the Grassmannian $\gr_d(\R^n)$ and substituting $T=\nu u(x)$, we deduce that
\begin{align*}
\E_V\left[\Big(\jac_\fp \left( \na_V w_V \right) \Big)^\fpp\right] &\leq \E_V \left[s^{-p}\,\Big(\det \left(\pi_H TT^* \pi_H\right)\Big)^{\frac{p}{2\fp}}\right]\\
&\leq \frac{ \G\left(\frac{n}{2}\right)\G\left(\frac{\gamma}{2}\right)}{\G\left(\frac{\fp+1}{2}\right)\G\left(\frac{n-p}{2}\right)} (n-1)^{-\frac{p}{2}} |\na u|_{\rm HS}^{p},
\end{align*}  
thanks to the independence of $H$ and $s$ (see Remark~\ref{rem: independence}), the uniform distribution of $H$,  the pointwise determinant bound~\eqref{pointwise, jac est}, as well as the linear algebraic Lemma~\ref{lem: linear algebra} and the statistical Lemma~\ref{lem: beta distribution}. This inequality holds for \textit{a.e.} $x \in \B^n$ and remains valid for $n=\fp+1$. Hence,
\begin{align}\label{vip, a}
&\int_{\gr_d(\R^n)}\int_{\B^n} \Big(\jac_\fp \left( \na_V w_V (x)\right) \Big)^\fpp \,\dd x\,\dd\mu_{\gr_d(\R^n)}(V)\nonumber\\
&\qquad \qquad \leq  \frac{ \G\left(\frac{n}{2}\right)\G\left(\frac{\fp+1-p}{2}\right)}{\G\left(\frac{\fp+1}{2}\right)\G\left(\frac{n-p}{2}\right)} (n-1)^{-\frac{p}{2}} \ep[u],
\end{align}
thanks to Fubini's theorem.

On the other hand, from the identity~\eqref{fubini} and Lemma~\ref{lem: wV}, we infer that
\begin{align*}
\int_{\B^n} \Big(\jac_\fp \left( \na_V w_V (x)\right) \Big)^\fpp \,\dd x = \int_{\B^{V^\perp}_1} \int_{\B^V_{\rho(y)}} \Big(\jac_\fp \left( \na w_{V,y} (\xi)\right) \Big)^\fpp \,\dd \xi\,\dd y.
\end{align*}
Recall $w_{V,y} \in W^{1,p}\left({\B^V_{\rho(y)}};\sph(V)\right)$ from Lemma~\ref{lem: wV}; also $\rho(y) = \sqrt{1-|y|^2}$. Since $d=\dim V=\fp+1$, we may apply Proposition~\ref{prop:sharp-k-jacobian} with $r=\rho(y)$ and $\gamma = \fp+1-p$ to deduce that
\begin{align*}
\int_{\B^V_{\rho(y)}} \Big(\jac_\fp \left( \na w_{V,y} (\xi)\right) \Big)^\fpp \,\dd \xi \geq \frac{\alpha_{\fp}}{\gamma}\rho(y)^\gamma.
\end{align*}
Hence,
\begin{equation}\label{xx}
\int_{\B^n} \Big(\jac_\fp \left( \na_V w_V (x)\right) \Big)^\fpp \,\dd x \geq \frac{\alpha_\fp}{\gamma} \int_{\B_1^{V^\perp}} (1-|y|^2)^{\frac{\gamma}{2}}\,\dd y.
\end{equation}
But $\B_1^{V^\perp}$ is the unit ball in $V^\perp \cong \R^{n-d} = \R^{n-1-\fp}$, so the integral on the right-hand side can be evaluated explicitly, by invoking various classical identities on special functions: 
\begin{align*}
\int_{\B_1^{V^\perp}} (1-|y|^2)^{\frac{\gamma}{2}}\,\dd y &= \alpha_{n-d-1} \int_0^1 r^{n-d-1} (1-r^2)^{\frac{\gamma}{2}}\,\dd r \\
&= \frac{\alpha_{n-d-1}}{2} B\left( \frac{n-d}{2}, 1+ \frac{\gamma}{2} \right)\\
&= \pi^{\frac{n-d}{2}} \frac{\G\left(1+\frac{\gamma}{2}\right)}{\G\left(1+\frac{n-d+\gamma}{2}\right)}.
\end{align*}
In addition, using the identities $\alpha_\fp = \frac{2\pi^{\frac{d}{2}}}{\G\left(\frac{d}{2}\right)}$, $\G(1+z)=z\G(z)$, $d=\fp+1$, and $\gamma=\fp+1-p$, we deduce from \eqref{xx} that
\begin{align}\label{vip, b}
&\int_{\gr_d(\R^n)}\int_{\B^n} \Big(\jac_\fp \left( \na_V w_V (x)\right) \Big)^\fpp \,\dd x\,\dd\mu_{\gr_d(\R^n)}(V)\nonumber\\
&\qquad\qquad \geq  \frac{2\pi^{\frac{n}{2}} \G\left( \frac{\fp+1-p}{2} \right)}{ (n-p) \G\left( \frac{\fp+1}{2}\right)  \G\left(\frac{n-p}{2}\right)},
\end{align}
where $\mu_{\gr_d(\R^n)}$ is the normalised Haar measure.

Therefore, putting together the bounds in~\eqref{vip, a} and \eqref{vip, b}, we arrive at \begin{align}\label{vip, c}
\ep[u] \geq (n-1)^{\frac{p}{2}} \frac{2\pi^{\frac{n}{2}}}{(n-p)\G\left(\frac{n}{2}\right)} = \frac{\alpha_{n-1}}{n-p}(n-1)^{\frac{p}{2}}.
\end{align}

For the hedgehog~$u_0(x) = \hedge$, it holds that
\begin{align*}
|\na u_0(x)|^p = (n-1)^{\frac{p}{2}} |x|^{-p}
\end{align*}
and hence that
\begin{align}\label{hedge hog p energy}
\ep[u_0] = (n-1)^{\frac{p}{2}} \alpha_{n-1} \int_0^1 r^{-p+n-1}\,\dd r =  \frac{\alpha_{n-1}}{n-p}(n-1)^{\frac{p}{2}}.
\end{align}
Therefore, we conclude from \eqref{vip, c} and \eqref{hedge hog p energy} that 
\begin{align*}
\inf_{u \in \mathscr{A}_{n,p}}\ep[u] = \ep[u_0].
\end{align*}
This proves the $p$-energy minimality of the hedgehog.

\subsection{Uniqueness}

Finally, we show that the hedgehog $u_0(x)=\hedge$ is the unique $\ep$-minimiser in $\anp$ by considering three separate ranges of indices: $p \in ]1,n-1[$, $p=n-1$, and $p \in ]n-1,n[$.

Upon the completion of the proof, an audit by GPT 5.6 ultra pro pointed out that once the $\ep$-minimality of the hedgehog is known, uniqueness follows directly from C. Wang~\cite[Theorem~2]{Wang1998}, for all $1<p<n$. Thus, the remaining parts of the paper should be read as an alternative proof of the uniqueness of $\ep$-minimiser.

\begin{proof}[Alternative proof of the uniqueness of minimiser]

In this proof, assume that $u$ is an $\ep$-minimiser over $\anp$. In view of~\S\ref{subsec: min}, we have $\ep[u]=\ep[u_0]$ and essentially all the inequalities in the preceding parts become equalities.

\smallskip
\noindent
{\bf Case~1: $1<p<n-1$.} In light of the case of equality in Lemma~\ref{lem: linear algebra}, we have 
\begin{equation}\label{T*T, uniqueness}
 \na u(x) (\na u(x))^* = \lambda(x) \,\id_{u(x)^\perp}
\end{equation} 
for \textit{a.e.} $x\in\B^n$. Take $V \in \gr_d(\R^n)$ and $H := V \cap u(x)^\perp$ as before. As in the previous arguments (see, \textit{e.g.}, the derivations leading to \eqref{pointwise, jac est} in \S\ref{subsec: min}), one has  the pointwise matrix inequality:
 \begin{equation}\label{leq, uniqueness}
 (\na u)\pi_V (\na u)^* \leq (\na u)(\na u)^*,
 \end{equation}
from which it follows that
 \begin{equation}\label{zz}
 \jac_\fp\left(\pi_H (\na u) \pi_V\right)\leq
 \jac_\fp\left(\pi_H \na u\right).
 \end{equation}
 Here and hereafter, we suppress the variable $x \in \B^n$ for notational simplicity.

Let us \emph{claim} that 
\begin{equation*}
\text{The equality in } \eqref{zz} \text{ holds if and only if } (\na u)^* H \subset V.\tag{V}
\end{equation*}
\begin{proof}[Proof of the claim~(V)]
The ``if'' part is immediate. For the ``only if'' part, we set 
\begin{align*}
 M:=\pi_H \na u \in {\rm Hom}\left(\R^n;H \right).
 \end{align*}
Observe by \eqref{T*T, uniqueness} that either $\lambda=0$ for some $x\in \B^n$, at which $\na u$ is also the zero matrix, or $M$ is of maximal rank. In the former case, the \emph{claim}~(V) holds trivially, so from now on we work with the latter case, namely that
\begin{equation}\label{M has maximal rank}
{\rm rank}(M)=\dim H = d-1=\fp.
\end{equation}

To proceed, in view of the pointwise matrix inequality~\eqref{leq, uniqueness}, one has  \begin{align}\label{M-tilde ineq}
 {\bf 0} \leq \widetilde{\M}:= (MM^*)^{-\frac{1}{2}} M\pi_V M^* (MM^*)^{-\frac{1}{2}} \leq \id_H.
 \end{align}
If the equality~\eqref{zz} holds, then $\det(M\pi_V M^*) = \det(MM^*)$ and hence $\det \widetilde{\M} = 1$. But the matrix inequality~\eqref{M-tilde ineq} implies that all the eigenvalues of $\widetilde{\M}$ lie in the interval $[0,1]$. Thus $\widetilde{\M}=\id_H$, so the pointwise identity follows:
\begin{equation}\label{MM* identity}
M\pi_V M^* = MM^*
\end{equation} 
as an element of ${\rm End}(H)$.

We may now conclude the proof. Indeed, \eqref{MM* identity} is equivalent to $M\left(\id_{\R^n}-\pi_V\right)M^*={\bf 0}.$ For any $v \in {\rm im}\left[\left(\id_{\R^n}-\pi_V\right)M^*\right]$, say $v=\left(\id_{\R^n}-\pi_V\right)M^*h$ for some $h\in H$, then from $Mv \cdot h = 0$ one deduces that $\left(\id_{\R^n}-\pi_V\right) (M^*h) \cdot (M^*h) =0$. Thus, $M^*h$ has zero component in the $V^\perp$-direction, \textit{i.e.}, $v=0$. This shows that
\begin{align*}
\left(\id_{\R^n}-\pi_V\right) M^* = {\bf 0}. 
\end{align*}
But $M^*=(\na u)^* \pi_H^*=(\na u)^* \pi_H$, so $(\na u)^*$ maps $H$ into $V$. This proves the \emph{claim}~(V).      \end{proof}

Now, by virtue of the \emph{claim}~(V), we have $(\na u)^*H \subset V$ for a randomly chosen $V \in \gr_d(\R^n)$, and $H:=V\cap u^\perp$. Here $\dim V =d < n$ and $O(n)$ acts transitively on   $\gr_d(\R^n)$, so 
\begin{equation}
(\na u)^* H \subset H\qquad \text{ for a.e. } H \subset u^\perp. 
\end{equation}
This implies that $(\na u)^* = \beta \, \pi_{u^\perp}$ and hence $$\na u = \beta \, \pi_{u^\perp}$$ for some function $\beta: \B^n \to \R$. As before, this is understood as for \textit{a.e.} $x \in \B^n$.

We proceed by an ODE argument. For \textit{a.e.} direction $\omega \in \sph^{n-1}$, consider $$u_\omega(r) := u(r\omega),$$ which by a slicing argument lies in $W^{1,p}\left(]\delta,1[;\sph^{n-1}\right)$ for \textit{a.e.} $\delta \in ]0,1[$. Note that $u_\omega(1) = \omega$ by the boundary condition for $u \in \anp$. It then satisfies that
\begin{align*}
u_\omega'(r) = \omega \cdot \na u(r\omega) = \beta(r\omega) \left[\omega - \left(u_\omega(r)\cdot\omega\right)u_\omega(r) \right].
\end{align*}
Thus, the variable $\widetilde{u}(r):=u_\omega(r)\cdot\omega$ satisfies
\begin{align*}
\widetilde{u}'(r) = \beta(r\omega) \left[1-\widetilde{u}(r)^2\right]
\end{align*}
with the ``initial'' datum $\widetilde{u}(1) = 1$. By the uniqueness of ODE with integrable coefficients, we find that $\widetilde{u} \equiv 1$ on $]\delta,1[$; that is,
\begin{align*}
u(r\omega) = u_\omega(r) = \omega \qquad \text{for a.e. } r\in ]\delta,1[ \text{ and } \omega \in \sph^{n-1}.
\end{align*}
Here $\delta>0$ is arbitrary, so we conclude by sending $\delta \searrow 0$ that
\begin{align*}
u(x) = \hedge\qquad\text{for a.e. } x \in \B^n.
\end{align*}
This proves the uniqueness of $\ep$-minimiser in $\anp$ whenever $1<p<n-1$.

\smallskip
\noindent
{\bf Case~2: $p=n-1$.} This essentially follows from the classical null-Lagrangian arguments. Indeed, from the computations in \S\ref{sec: lower bd} we deduce that 
\begin{align*}
\alpha_{n-1} = \int_{\B^n} u \cdot \jj\,\dd x &\leq \int_{\B^n}|\jj|\,\dd x \\
&= \int_{\B^n} \left|\jac_{n-1}(\na u)\right|\,\dd x \leq (n-1)^{-\frac{n-1}{2}}\int_{\B^n}|\na u|^{n-1}\,\dd x
\end{align*}
for $\jj := \star\left(u^\#\omega_{\sph^{n-1}}\right)$. Also, by the minimality of hedgehog established in \S\ref{subsec: min}, the right-most term equals $\alpha_{n-1}$ when $u$ is an $\ep$-minimiser in $\anp$, so equalities must hold everywhere. Noting that $\jj$ spans the kernel of $\na u$ at points where $\na u$ has rank $(n-1)$, we thus have
\begin{equation}
(\na u) \cdot u = \na_u u = 0\qquad\text{a.e. on }\B^n.
\end{equation}
 
It is well-known, \textit{cf. e.g.}, Hardt--Lin--Wang~\cite{HardtLinWang1997}, that $u$ is regular outside a finite set away from the boundary. For any given $\omega \in \sph^{n-1}$, shoot an inward characteristic curve $\gamma$ from $\omega$:
\begin{equation*}
\begin{cases}
\gamma'(t) = -u\left(\gamma(t)\right),\\
\gamma(0)=\omega.
\end{cases}
\end{equation*}
Then, $\gamma$ has no acceleration:
\begin{align*}
\gamma''(t) = -\na u(\gamma(t))\cdot u(\gamma(t)) \equiv 0.
\end{align*}
Together with the boundary condition $u(\gamma(0))=u(\omega)=\omega$, this implies that $\gamma(t) = (1-t)\omega$ and hence $u(t\omega)\equiv \omega$. The above argument is rigorous whenever the characteristic curve $\gamma$ avoids the finite singular set of $u$. We may thus conclude that $u(x)=\hedge$ for \textit{a.e.} $x \in \B^n$.

\smallskip
\noindent
{\bf Case~3: $n-1<p<n$.} In this case, note that $\fp=n-1$, $d=\fp+1=n$, and $\gamma = \fp+1-p = n-p$. Again, from the computations in \S\ref{sec: lower bd}, one deduces that 
\begin{align}\label{case 3, eq}
\ep[u] = (n-1)^{\frac{p}{2}}\int_{\B^n}\left|\jac_{n-1}(\na u)\right|^{\frac{p}{n-1}}\,\dd x = (n-1)^{\frac{p}{2}}\frac{\alpha_{n-1}}{n-p}.
\end{align}
As in the previous case, $u$ is regular in a neighbourhood of the boundary, and it has only finitely many isolated singularities. A degree-0 isolated singularity is removable, so we assume that all the singularities have nonzero integer degrees.

We check that $u$ has only one singularity; say $s \in \B^n$. Indeed, were there at least two singularities, by an inspection on the ``colliding-and-merging ball'' argument in the proof of Proposition~\ref{prop:off-centre}, the inequality~\eqref{credit, xx} would be strict since the power $\gamma = n-p$ is strictly less than $1$. This would lead to $\int_{\B^n}\left|\jac_{n-1}(\na u)\right|^{\frac{p}{n-1}}\,\dd x > \frac{\alpha_{n-1}}{n-p}$. Contradiction.

Let us introduce the ``defect'':
\begin{align*}
\D_s(r) := \int_{\B^n_r(s)} \left|\jac_{n-1}\left(\na \widetilde{u}\right)\right|^{\frac{p}{n-1}}\,\dd x - \frac{\alpha_{n-1}}{n-p} r^{n-p},
\end{align*}
where $\widetilde{u} = u$ on $\B^n$ and $\widetilde{u}(x)=\hedge$ on $\R^n \setminus \B^n$. Clearly $\D_s(r)$ is increasing in $r>0$, and by \S\ref{sec: lower bd} it holds that $\D_s(r) \geq 0$ for any $r >0$. Also, in view of \eqref{case 3, eq} and arguments in  \S\ref{sec: lower bd}, we have $\lim_{r \to \infty} \D_s(r)=0$. Therefore, $\D_s(r)\equiv 0$ for all $r>0$.

Now, for a sphere $\Sigma$, denote by $\na_\tau$ the tangential gradient, \emph{i.e.}, the projection of the gradient to $T\Sigma$. By H\"{o}lder's inequality, we have that
\begin{align}\label{chain of ineq}
\int_{\p\B^n_r(s)}\left|\na \widetilde{u}\right|^p\,\dd\hau^{n-1} &\geq (n-1)^{\frac{p}{2}} \int_{\p\B^n_r(s)} \jac_{n-1}\left(\na_\tau\widetilde{u}\right)^{\frac{p}{n-1}}\,\dd\hau^{n-1}\nonumber\\
&\geq (n-1)^{\frac{p}{2}} \left[\hau^{n-1}\left(\p\B^n_r(s)\right)\right]^{1-\frac{p}{n-1}} \left\{\int_{\p\B^n_r(s)} \jac_{n-1}\left(\na_\tau\widetilde{u}\right)\,\dd\hau^{n-1}\right\}^{\frac{p}{n-1}}\nonumber\\
&\geq (n-1)^{\frac{p}{2}} \left[ r^{n-1}\alpha_{n-1}\right]^{1-\frac{p}{n-1}}  \left|\int_{\p\B^n_r(s)} \det\left(\na_\tau\widetilde{u}\right)\,\dd\hau^{n-1}\right|^{\frac{p}{n-1}}\nonumber\\
&=  (n-1)^{\frac{p}{2}} \left[ r^{n-1}\alpha_{n-1}\right]^{1-\frac{p}{n-1}}\left( \alpha_{n-1} \left|\deg\left(\widetilde{u};s\right)\right|\right)^{\frac{p}{n-1}}\nonumber\\
&\geq (n-1)^{\frac{p}{2}} \alpha_{n-1} r^{n-p-1}.
\end{align}
But an integration over $r$ of the above inequalities leads to 
\begin{align*}
\int_{\B^n_R(s)}|\na u|^p\,\dd x \geq \frac{(n-1)^{\frac{p}{2}}\alpha_{n-1}}{n-p} R^{n-p},
\end{align*}
where the equality holds since $\D_s(R)=0$ for any $R>0$. Thus, equalities must hold everywhere in the chain of inequalities~\eqref{chain of ineq} for \textit{a.e.} $r>0$.

From the equal cases in~\eqref{chain of ineq}, we deduce that the singularity $s$ is of degree $1$, and that
\begin{equation}
\jac_{n-1}\left(\na \widetilde{u}\right) = \jac_{n-1}\left(\na_\tau\widetilde{u}\right)\qquad\text{at a.e. $x \in \p\B^n_r(s)$ and for a.e. $r>0$.} 
\end{equation}
 By linear algebra, this implies that 
\begin{equation}\label{foliation}
\na  \widetilde{u}(x) \cdot \frac{x-s}{r} =0\qquad\text{for a.e. $x \in \p\B^n_r(s)$}, 
\end{equation}  
so the foliation by spheres centred at $s$ is everywhere tangent to the level sets of $\widetilde{u}$. But outside $\B^n$, the map $\widetilde{u} = \hedge$, so we must have $s=0$. 
Then~\eqref{foliation} in turn yields that $\widetilde{u}$ is radial. Hence it coincides with the hedgehog on $\B^n$.

The proof for the uniqueness of $\ep$-minimiser is now complete.  \end{proof}

\appendix
\section*{Appendix: Proof of the approximation lemma}

\addcontentsline{toc}{section}{Appendices}
\renewcommand{\thesubsection}{\Alph{subsection}}

 The argument presented below is standard. It is an adaptation of, for instance,  Bethuel~\cite{Bethuel1991}.

\begin{proof}[Proof of Lemma~\ref{lem: approx}]

We divide our arguments into three steps below.
All the balls in this proof are $(k+1)$-dimensional, so we denote systematically $\B_R(x_0) \equiv \B^{k+1}_R(x_0)$.

\smallskip
\noindent
{\bf Step 1.} Fix $\eta>0$ and extend $u$ to $ \B_{r+\eta}$ by
\[
 U(x)=
 \begin{cases}
 u(x)\qquad\text{ for } |x|<r,\\
 \hedge\qquad\text{ for } r<|x|<r+\eta.
 \end{cases}
\]
Since the trace of $u$ equals $\hedge$ on $\partial \B_r$, we have $U\in W^{1,p}(\B_{r+\eta};\sph^k)$. 
Choose $\lambda_\ell\searrow 1$ with $\lambda_\ell r<r+\eta$ and define
\[
        u^{(\ell)}(x)=U(\lambda_\ell x),
        \qquad x\in \B_r.
\]
Strong continuity of dilations in $W^{1,p}$ then implies that $u^{(\ell)}\rightarrow u$ strongly in $W^{1,p}(\B_r)$. Moreover,
\[
        u^{(\ell)}(x)= \hedge
        \qquad\text{whenever }\frac r{\lambda_\ell}<|x|<r.
\]
By passing to the limit, it suffices to prove for a map $v$ such that $v=\hedge$ on $\{r-\delta<|x|<r\}$ for some $\delta>0$.  

\smallskip
\noindent
{\bf Step 2.} Choose $g\in C^\infty\left(\overline{\B_r};\R^{k+1}\right)$ such that $g=\hedge$ on $\{r-\delta<|x|<r\}$. For example, one may take $g(x)=\varpi(|x|)\frac{x}{|x|}$, where $\varpi$ vanishes in a neighbourhood of the origin and is
identically one for $|x|>r-\delta$.  The map $(v-g)$ then vanishes on a neighbourhood of $\partial \B_r$. Extending it by zero and mollifying, we obtain a sequence
 $\{f_j\}\subset C_c^\infty(\B_r;\mathbb R^{k+1})$ such that $v_j:=g+f_j$ lies in $ C^\infty(\overline{\B_r};\mathbb R^{k+1})$, and $v_j \to v$ strongly in $W^{1,p}$. In addition, by choosing the mollification scales sufficiently small, we have $v_j = \hedge$ on a fixed smaller boundary collar. 

\smallskip
\noindent
{\bf Step 3.} Now we choose $ \chi\in C_c^\infty(\B_{1/2})$ such that $\chi\equiv 1$ on $\B_{1/4}$.
Fix $\varepsilon>0$ so small that $ \varepsilon<\frac14$ and $\varepsilon\|\na \chi\|_{L^\infty}<\frac12$. Then, for $a\in \B_\varepsilon$, we set
\[
        F_a(z)=z-\chi(z)a.
\]
Observe that
\[
        \|\na F_a-\id\|\leq |a|\,\|\na\chi\|_\infty<\frac12,
\]
so is a smooth biLipschitz diffeomorphism on
$\mathbb R^d$. Moreover, $F_a(z)=0$ if and only if $z=a$, and $|F_a(z)| =|F_a(z)-F_a(a)|\geq\frac12|z-a|$.

To proceed, let us define
\[
        P_a(z)=\frac{F_a(z)}{|F_a(z)|}\qquad
        \text{for any } z\neq a.
\]
It lies in $C^\infty(\mathbb R^{k+1}\setminus\{a\};\sph^{k})$. Meanwhile,
\begin{equation}
        |\na P_a(z)|\leq\frac{C}{|z-a|}
        \label{eq:moving-projection-gradient}
\end{equation}
and, since $F_a(z)=z$ for $|z|\geq1/2$, 
\begin{equation}
        P_a(z)=\frac z{|z|}
        \qquad\text{whenever }|z|\geq\frac12.
        \label{eq:moving-projection-normalization}
\end{equation}
In addition, put \[
        E_j=\{x\in B_r:|v_j(x)|<1/2\}.
\]
Since $|v|=1$ \textit{a.e.} and $v_j\to v$ strongly in $W^{1,p}$, we have $\leb^{k+1}(E_j) \to 0$ and
\begin{align}
 \int_{E_j}|\na v_j|^p\,\dd x
 &\leq C\int_{E_j}|\na v_j-\na v|^p\,\dd x
      +C\int_{E_j}|\na v|^p\,\dd x
 \longrightarrow 0.
 \label{eq:bad-set-energy}
\end{align}

By \eqref{eq:moving-projection-gradient}, Fubini's theorem, and
$p<k+1$, we obtain that
\begin{align}
 &\int_{\B_\varepsilon}
       \int_{E_j}|\na (P_a\circ v_j)|^p\,\dd x\,\dd a \notag\\
 &\qquad\leq
 C\int_{E_j}|\na v_j(x)|^p
       \left(
       \int_{\B_\varepsilon}
       \frac{\dd a}{|v_j(x)-a|^p}
       \right)\dd x \notag\\
 &\qquad\leq
 C\varepsilon^{d-p}
 \int_{E_j}|\na v_j|^p\,\dd x.
 \label{eq:averaged-projection}
\end{align}
The last line follows from the uniform estimate
\[
        \sup_{z\in\mathbb R^{k+1}}
        \int_{\B_\varepsilon^{k+1}}|z-a|^{-p}\,\dd a
        \leq C(n,k)\varepsilon^{k+1-p}\qquad\text{ if } p<k+1.
\]
By Sard's theorem, almost every $a\in \B_\varepsilon$ is a regular
value of $v_j$. We may therefore choose a regular value $a_j$ such
that
\begin{equation}
        \int_{E_j}|\na(P_{a_j}\circ v_j)|^p\,\dd x
        \leq
        C\int_{E_j}|\na v_j|^p
        \longrightarrow0.
        \label{eq:selected-projection}
\end{equation}

Now, define
\[
        w_j=P_{a_j}\circ v_j
\]
away from $\Sigma_j:=v_j^{-1}(a_j)$, and define it arbitrarily at those points. Since $a_j$ is a regular value and $v_j=\hedge$ near the boundary, $\Sigma_j$ is finite set in the interior of $\B_r$. Furthermore, $        w_j\in C^\infty\left(\overline{\B_r}\setminus\Sigma\right)$. Near each point $s\in\Sigma_j$, the inverse function theorem and \eqref{eq:moving-projection-gradient} give us 
\[
        |\na w_j(x)|\leq \frac{C}{|x-s|},
\]
which is locally $L^p$ as $p<k+1$. Points have zero $p$-capacity
in dimension $k+1$, so $w_j$ can be extended across $\Sigma_j$ as a Sobolev map: without relabelling, we obtain $w_j\in W^{1,p}\left(\B^{k+1}_r;\sph^{k}\right).$ In addition, on the boundary collar, $v_j=\hedge$. Hence
\eqref{eq:moving-projection-normalization} implies that 
\[
        w_j=P_{a_j}\left(\hedge\right)=\hedge.
\]

It remains to verify strong convergence. On $\B_r\setminus E_j$, one infers from~\eqref{eq:moving-projection-normalization} that 
\[
        w_j=\frac{v_j}{|v_j|}.
\]
As normalisation is smooth on $\{|z|\geq1/2\}$ and $v_j\to v$ strongly in $W^{1,p}$ with $|v|=1$, it follows that $w_j \to v$ strongly in $ W^{1,p}(\B_r\setminus E_j)$. On $E_j$, boundedness of the maps and the convergence results~\eqref{eq:bad-set-energy} and
\eqref{eq:selected-projection} yield that
\[
 \int_{E_j}|w_j-v|^p\,\dd x+\int_{E_j}|\na w_j-\na v|^p \longrightarrow 0.
\]
Thus $w_j \to v$  strongly in $ W^{1,p}(\B_r)$.

We conclude the proof of the approximation lemma by combining this construction with Step~1 and taking a diagonal
 subsequence.     \end{proof}

\medskip
\noindent
{\bf Acknowledgement}. The research of SL is supported by NSFC Projects 12201399, 12331008, and 12411530065, Young Elite Scientists Sponsorship Program by CAST 2023QNRC001, National Key Research $\&$ Development Programs 2023YFA1010900 and 2024YFA1014900, Shanghai Rising-Star Program 24QA2703600, Qi-Guang Scholarship, and Shanghai Frontier Research Institute for Modern Analysis.

	\noindent
	{\bf Statement of competing interests}. 
	The author declares that there is no conflict of interest.

	\noindent
	{\bf Statement of data availability}.
	Our manuscript has no associated data.

	\noindent
	{\bf AI Statement}. During preparation of this manuscript, the author used ChatGPT 5.6 sol model for exploratory discussion and language editing. The author independently verified the mathematical content and assumes full responsibility for it. 

\end{document}